%% file: presentation-algebras-in-chain-complexes.tex
\documentclass[a4paper]{amsart}

\pdfoutput=1 

\usepackage{cmbright}
\usepackage{mathrsfs}
\usepackage{xcolor}
\usepackage[utf8]{inputenc}
\usepackage[T1]{fontenc}
\usepackage[english]{babel}
\usepackage{amsfonts}
\usepackage{amsthm}
\usepackage{amssymb}
\usepackage{mathcomp}
\usepackage{mathrsfs}
\usepackage[bb=boondox]{mathalfa}
\usepackage{graphicx}
\usepackage{epstopdf}
\usepackage[centertags]{amsmath}
\usepackage[pdftex]{hyperref}
\usepackage{vmargin}
\usepackage{tikz}
\usepackage{tikz-cd}
\usepackage{dsfont}
\usepackage{cases}
\usepackage{mathtools}
\usepackage{shuffle}
\usepackage{adjustbox}

\include{macros}

\title{Presentations of algebras in chain complexes}
\author{Brice Le Grignou}
\email{bricelegrignou "at" gmail.com}

\date{\today}

\begin{document}

\begin{abstract}
    The goal of this article is to describe several presentations of the infinity category of algebras over some monad on the infinity category of chain complexes.
\end{abstract}
\maketitle

\setcounter{tocdepth}{1}
\tableofcontents

\section*{Introduction}

Many categories encountered in Algebra are presentable. This means that for such a category $\categ C$, one can find a small category $\categ D$ and an equivalence relating $\categ C$ to the full subcategory of $\Fun{\categ D^\op}{\categ{Set}}$ spanned by functors $F : \categ D \to \categ{Set}$ that satisfy some locality conditions.

Finite products theories also called Props provides many examples of such presentable categories in Algebra.  For instance, if $R$ is a commutative ring, the category of $R$-modules may be presented as follows. Let $\categ{Cart}_R$ be the full subcategory of $R$-modules spanned by those of the form $R^n$ for $n \in \mathbb N$. Then the category of $R$-modules is equivalent to the full subcategory of $\Fun{\categ{Cart}_R^\op}{\categ{Set}}$ spanned by functors that preserves finite products.

Moreover, since $R$-modules are equivalent to such products preserving functors, then many categories of algebras in $R$-modules have a similar presentation. Indeed, let us consider a monadic adjunction
$$
\begin{tikzcd}
R-\categ{Modules} \ar[rr, shift left, "T_M"]
&& \categ A \ar[ll, shift left, "U^M"]
\end{tikzcd}
$$
whose associated monad $M$ preserves reflective coequalisers and filtered colimits, and let $\categ{Cart}_M$ be the full subcategory of $\categ A$ spanned by objects of the form $T_M(R^n)$ for $n \in \mathbb N$. Then the category $\categ A$ is equivalent to the full subcategory of $\Fun{\categ{Cart}_M^\op}{\categ{Set}}$ spanned by functors that preserves finite products.

These results have extensions in the world of $\infty$-categories. Indeed, if $\categ{Ch}_R^{\geq 0}$ is the $\infty$-category of chain complexes of $R$-modules in non negative degrees and if $\categ{S}$ is the $\infty$-category of $\infty$-groupoids, then we have the following result.

\begin{proposition*}\cite[Corollary 4.7.3.18]{Lurie18}
Let $U^M : \categ A \to \categ{Ch}_R$ be a monadic functor that preserves filtered colimits and geometric realisations and let $T_M$ be its left adjoint. If $\categ{Cart}_M$ is the full subcategory of $\categ A$ spanned by the objects $T_M(R^n), n \in \mathbb N$, then $\categ A$ is equivalent to the full subcategory of $\Fun{\categ{Cart}_M^\op}{\categ S}$ spanned by functors that preserve finite products.
\end{proposition*}

The goal of this article is to get similar presentations in the stable context. More precisely, given a monadic adjunction relating the product $\categ{Ch}_R^K$ over a set $K$ of the $\infty$-category of chain complexes $\categ{Ch}_R$
to an $\infty$-category $\categ A$
$$
\begin{tikzcd}
\categ{Ch}_R^K \ar[rr, shift left, "T_M"]
&& \categ{A} \ar[ll, shift left,"U^M"]
\end{tikzcd}
$$
whose right adjoint $U^M$ preserves filtered colimits and geometric realisations, we describe several presentations of the $\infty$-category $\categ A$.

One kind of presentation is given by finitely cellular objects, that is objects that are obtained from the initial object by a finite composition of pushouts of the form
$$
\begin{tikzcd}
T_M (S^m_k)
\ar[r] \ar[d]
& A
\ar[d]
\\
T_M(0)
\ar[r]
& B
\end{tikzcd}
$$
for some $k \in K$ and $m \in \mathbb Z$, where $S^m_k$ denotes the object of $\categ{Ch}_R^K$ whose $k$-component is the shift of $R$ to the $m^{th}$ degree and whose $k'$-component, for $k' \neq k$, is the zero chain complex. For any $n \in \mathbb Z \sqcup \{\infty\}$, we denote $\categ{Cell}_{M,\leq n}$ the full subcategory of $\categ A$ spanned by finitely cellular objects, with the additional restriction that they are built using a finite composition of pushout as above but so that $m \leq n-1$.

\begin{theorem*}[\ref{theoremcellularpresentation}]
For any $n \in \mathbb Z \sqcup \{\infty\}$
the $\infty$-category $\categ A$ is canonically equivalent to the full subcategory of $\Fun{\categ{Cell}_{M,\leq n}^\op}{\categ S}$ spanned by functors $F$ so that
\begin{enumerate}
    \item $F(\emptyset) \simeq \ast$;
    \item for any pushout square in $\categ{Cell}_{M,\leq n}$ as above
its image through $F$ is a pullback.
\end{enumerate}
\end{theorem*}

Another type of presentation, that we call cartesian presentation is closer to the presentations shown in the unstable case.
For any $n \in \mathbb Z \sqcup \{\infty\}$, we denote $\categ{Cart}_{M,\leq n}$ the full subcategory of $\categ A$ spanned by finite coproducts of objects of the form $T_M(S^m_k)$ for $m \leq n$ and $k \in K$.

\begin{theorem*}[\ref{theoremcartesianpresentation}]
For any $n \in \mathbb Z \sqcup \{\infty\}$
the $\infty$-category $\categ A$ is canonically equivalent to the full subcategory of $\Fun{\categ{Cart}_{M,\leq n}^\op}{\categ S}$ spanned by functors $F$ so that
\begin{enumerate}
    \item $F$ preserves finite products;
    \item for any pushout square in $\categ{Cart}_{M,\leq n}$ of the form
    $$
\begin{tikzcd}
T_M (X)
\ar[r] \ar[d]
& T_M(0)
\ar[d]
\\
T_M(0)
\ar[r]
& T_M(\Sigma X)
\end{tikzcd}
$$
its image through $F$ is a pullback.
\end{enumerate}
\end{theorem*}

These two theorems are direct consequences of a single result. Indeed, let us consider a monadic functor
$U^M$ from an $\infty$-category $\categ A$ to an $\infty$-category $\categ C$ that preserves sifted colimits.
We assume that $\categ C$ is $\omega$-presentable in the sense that there exists a small $\infty$-category $\categ D$ so that $\categ C$ is a reflective subcategory
of $\Fun{\categ D^\op}{\categ S}$ stable through filtered colimits.
Thus we have a sequence of adjunctions
$$
\begin{tikzcd}
\Fun{\categ D^\op}{\categ S} \ar[rr, shift left, "i_!"]
&& \categ C \ar[ll, shift left, "i^\ast"]
\ar[rr, shift left, "T_M"]
&& \categ A . \ar[ll, shift left, "U^M"]
\end{tikzcd}
$$
Then, let us consider a full subcategory $j : \categ B \to \categ A$ so that the composite functor from $\categ D$ to $\categ A$
factorises through $j$. Let us also suppose that $\categ B$ contains only $\omega$-small objects and that it is stable through
finite coproducts. We obtain a square of right adjoint functors
$$
\begin{tikzcd}
\Fun{\categ B^\op}{\categ S}
\ar[r,"t^\ast"]
&\Fun{\categ D^\op}{\categ S}
\\
\categ A
\ar[u, "j^\ast"] \ar[r,swap, "U^M"]
& \categ{C}
\ar[u, "i^\ast"]
\end{tikzcd}
$$
Let $\categ E$ be a reflective subcategory of $\Fun{\categ B^\op}{\categ S}$. If $\categ E$ is stable through filtered colimits, if its functors preserve finite products and if it contains the image of $\categ A$ in $\Fun{\categ D^\op}{\categ S}$, then one has the following result.

\begin{theorem*}[\ref{thmmain}]
The following assertions are equivalent.
\begin{enumerate}
    \item the functor from $\categ A$ to $\categ E$ is an equivalence;
    \item the composite functor
    $$
\categ E \xrightarrow{q} \Fun{\categ B^\op}{\categ S} \xrightarrow{t^\ast} \Fun{\categ D^\op}{\categ S}
$$
is conservative and sends elements of $\categ E$ to the essential image of the functor $i^\ast$ from $\categ C$ to $\Fun{\categ D^\op}{\categ S}$;
    \item the composite functor
$$
\categ E \xrightarrow{q} \Fun{\categ B^\op}{\categ S} \xrightarrow{t^\ast} \Fun{\categ D^\op}{\categ S} \xrightarrow{i_!} \categ C
$$
is conservative.
\end{enumerate}
\end{theorem*}

Finally, a recurring context where one can apply these results is provided by the category of algebras over a $K$-coloured operad $\operad P$ enriched in chain complexes of $R$-modules. Indeed, if $\operad P$ is $\Sigma$-cofibrant and if the model structure on products of chain complexes may be transferred to the category of $\categ P$-algebras along the adjunction
$$
\begin{tikzcd}
\categ{Chain}_R^K 
\ar[rr, shift left ,"T_P"]
&&\catofalg{\categ P}
\ar[ll, shift left ,"U^P"]
\end{tikzcd}
$$
then the resulting adjunction between the $\infty$-categories obtained from localisation at weak equivalences is monadic and its right adjoint preserves filtered colimits and geometric realisations.

\subsection*{Layout}

In the first section, we recall definitions and results about presentable $\infty$-cateogories and monadic functors. Moreover, we prove Theorem \ref{thmmain}. The second section deals with chain complexes. We describe there presentations of the $\infty$-categories of chain complexes of modules over a ring $R$ and of some $\infty$-categories of algebras in chain complexes. This latter presentations correspond to Theorem \ref{theoremcellularpresentation} and
Theorem \ref{theoremcartesianpresentation}. Finally the appendix gathers some older results about chain complexes in non negative degrees.

\subsection*{Acknowledgment}

I would like to thank Edouard Balzin for his comments and Gabriele Vezzosi for presenting me the article \cite{Lurie11bis} which has inspired this paper.

\subsection*{Notations and conventions}

\begin{itemize}
\itemt For us, $\infty$-category will mean locally small $(\infty,1)$-category. Since we mainly use references from the books Higher Topos Theory \cite{Lurie17} and Higher Algebra \cite{Lurie18} by Lurie, we will assume that $\infty$-categories are encoded using quasi-categories.
\itemt We consider two universes $\mathcal U < \mathcal V$. Small will mean $\mathcal U$-small. Sets and $\infty$-groupoids will be small while many categories and $\infty$-categories considered will be $\mathcal U$-large, but still $\mathcal V$-small.
    \itemt $\categ S$ will denote the $\infty$-category of groupoids (given for instance by the simplicial nerve of the simplicial category of Kan complexes).
    \itemt The first infinite cardinal will be denoted $\omega$. We will use the denomination $\omega$-small and $\omega$-presentable instead of compact and compactly generated. Moreover, a functor that preserves filtered colimits will be called $\omega$-accessible.
    \itemt As in \cite{Lurie17}, geometric realisation means colimit of a simplicial diagram.
    \itemt In a stable $\infty$-category, the loop functor is denoted $\Omega$ and the suspension functor is denoted $\Sigma$.
\end{itemize}

\section{Presentable infinity-categories and algebras}

In this section, we recall some definitions and results about presentable $\infty$-categories and monadic adjunctions. Then, we prove our main result that gives conditions for an $\infty$-category to present an $\infty$-category of algebras in another $\infty$-category equipped with a presentation.

\subsection{Presheaves infinity categories and their subcategories}

\begin{definition}
Given a small $\infty$-category $\categ D$, its presheaves $\infty$-category $\Psh(\categ D)$ is the mapping $\infty$-category from $\categ D^{\op}$ to $\categ S$
$$
\Psh(\categ D) = \Fun{\categ D^{\op}}{\categ S} .
$$
\end{definition}

\begin{proposition}\cite[Theorem 5.1.5.6, Corollary 5.5.2.9, Remark 5.5.2.10]{Lurie17}
The presheaves $\infty$-category $\Psh(\categ D)$ is the free cocomplete $\infty$-category built from $\categ D$ in the sense that it is cocomplete and for any cocomplete $\infty$-category $\categ C$, the two functors
$$
\mathrm{Fun}^{\mathrm{L}}\left({\Psh(\categ D)},{\categ C}\right) \hookrightarrow
\mathrm{Fun}^{\mathrm{cc}}\left({\Psh(\categ D)},{\categ C}\right) \to \Fun{\categ D}{\categ C}
$$
are both equivalences where $\mathrm{Fun}^{\mathrm{L}}\left({\Psh(\categ D)},{\categ C}\right)$ and $
\mathrm{Fun}^{\mathrm{cc}}\left({\Psh(\categ D)},{\categ C}\right)$ denote respectively the full subcategories of $\Fun{\Psh(\categ D)}{\categ C}$ spanned respectively by left adjoint functors and functors that preserve small colimits.
\end{proposition}

\begin{proposition}\cite[Corollary 5.1.2.3]{Lurie17}
The presheaves $\infty$-category $\Psh(\categ D)$ is also complete. Moreover, the limits and colimits are computed objectwise in the sense that for any object $d \in \categ D$, the evaluation at $d$ functor
$$
\Psh(\categ D) = \Fun{\categ D^{\op}}{\categ S}\xrightarrow{F \mapsto F(d)} \categ S
$$
preserves limits and colimits.
\end{proposition}

\begin{definition}
Let $\kappa$ be a regular cardinal. The $\kappa$ ind completion $\Ind_{\kappa}(\categ D)$ of a small $\infty$-category $\categ D$ is the full subcategory of $\Psh(\categ D)$ spanned by functors $F : \categ D^{\op}\to \categ S$ whose Grothendieck construction (that is the source of the right fibration that classifies this functor) is a filtered $\infty$-category.
\end{definition}

\begin{proposition}\cite[Corollary 5.3.5.4]{Lurie17}
A functor $F \in \Psh(\categ D)$ belongs to $\Ind_{\kappa}(\categ D)$ if and only if there exists a $\kappa$-filtered diagram $\categ I \to \categ D$
so that $F$ is the colimit in $\Psh(\categ D)$ of the diagram
$$
\categ I \to \categ D \hookrightarrow \Psh(\categ D).
$$
\end{proposition}

\begin{proposition}\cite[Corollary 5.3.5.4]{Lurie17}
If the small category $\categ D$ admits $\kappa$-small colimits, then a functor $F \in \Psh(\categ D)$ belongs to $\Ind_{\kappa}(\categ D)$ if and only if it preserves $\kappa$-small colimits.
\end{proposition}

\begin{definition}
Let $\categ D$ be a small $\infty$-category that has finite coproducts. Then, let $\Psh_{\Sigma}(\categ D)$ be the full subcategory of $\Psh(\categ D)$ spanned by functors $F$ that preserves finite products.
\end{definition}

\begin{proposition}\cite[Proposition 5.5.8.10]{Lurie17}
The inclusion functor
$\Psh_{\Sigma}(\categ D) \to \Psh(\categ D)$ preserves sifted colimits and
is right adjoint. In particular, $\Psh_{\Sigma}(\categ D)$ is cocomplete.
\end{proposition}

\begin{proposition}\cite[Lemma 5.5.8.14]{Lurie17}
A functor $F \in \Psh(\categ D)$ belongs to 
$\Psh_{\Sigma}(\categ D)$ if and only if it obtained as the geometric realisation in $\Psh(\categ D)$ of
a simplicial object $G$ whose simplicies belong to $\Ind_{\omega}(\categ D)\subseteq \Psh(\categ D)$.
\end{proposition}

\subsection{Presentable infinity-category}

\begin{definition}
Let $\kappa$ be an regular cardinal. An $\infty$-category $\categ C$ is $\kappa$-presentable if there exists a small $\infty$-category $\categ D$ and an adjunction
    $$
    \begin{tikzcd}
    \Psh(\categ D) \ar[rr, shift left]
    && \categ C \ar[ll, shift left]
    \end{tikzcd}
    $$
    whose right adjoint is full faithful and preserves $\kappa$-filtered colimits.
\end{definition}

\begin{remark}
The definition of a $\kappa$-presentable category is different from that \cite{Lurie17}. But, the fact that both definitions describe the same notion follows from \cite[Corollary 5.5.7.3]{Lurie17}.
\end{remark}

\begin{lemma}\cite[Proposition 5.5.4.15]{Lurie17}\label{lemmapresentatio}
Let $\kappa$ be a regular cardinal. Let $\categ D$ be a small category and let $S$ be a small set of morphisms of $\Psh(\categ D)$ between $\kappa$-small objects. Then the inclusion of the full subcategory of $\Psh(\categ D)$ spanned by $S$-local objects into $\Psh(\categ D)$ preserves $\kappa$-filtered colimits and is right adjoint. Hence, this full subcategory is $\kappa$-presentable.
\end{lemma}

\begin{proof}
From \cite[Proposition 5.5.4.15]{Lurie17} that uses the small object argument, we get the left adjoint (together with some accessibility condition on the right adjoint inclusion). It suffices to show that the full subcategory of $S$-local objects is stable through $\kappa$-filtered colimits (which follows from the same arguments that those used to prove the accessibility condition). For any $\kappa$-filtered diagram $D : K \to \Psh(\categ D)$ which takes values in the full subcategory of $S$-local objects and for any map $s : X \to Y$ in $S$, let us consider the following commutative square
$$
\begin{tikzcd}
\varinjlim \Mapp{\Psh(\categ D)}{Y}{D}
\ar[r] \ar[d]
& \varinjlim \Mapp{\Psh(\categ D)}{X}{D}
\ar[d]
\\
\Mapp{\Psh(\categ D)}{Y}{\varinjlim D}
\ar[r]
& \Mapp{\Psh(\categ D)}{X}{\varinjlim D}.
\end{tikzcd}
$$
Since $D$ takes values in $S$-local objects, the top horizontal arrow is an equivalence. Since $X$ and $Y$ are $\kappa$-small, the vertical arrows are equivalences. Hence the bottom horizontal map is also an equivalence. Thus the colimit of $D$ is $S$-local. This shows that the subcategory of $S$-local objects is stable through $\kappa$-filtered colimits and so the inclusion functor preserves these $\kappa$-filtered colimits.
\end{proof}

\begin{example}
Let $\categ D$ be a small $\infty$-category with finite coproducts. Then $\Psh_\Sigma(\categ D)$ is $\omega$-presentable and the inclusion functor
$$
\Psh_\Sigma(\categ D) \to \Psh(\categ D)
$$
is right adjoint and preserves filtered colimits.
\end{example}

\begin{theorem}\cite[Theorem 5.5.1.1, Corollary 5.5.2.4]{Lurie17}, \cite{Simpson}
A presentable $\infty$-category is complete and cocomplete and the full subcategory of its $\lambda$-small objects is small for any small cardinal $\lambda$.
\end{theorem}

Let us end this subsection by dealing with stable $\omega$-presentable $\infty$-categories.
Let us consider a stable cocomplete $\infty$-category $\categ C$, and a small  full subcategory $\categ D$ of $\categ C$ spanned by $\omega$-small objects. This yields
an adjunction
$$
\begin{tikzcd}
\Psh(\categ D)
\ar[rr, shift left]
&&
\categ C . \ar[ll, shift left]
\end{tikzcd}
$$
whose right adjoint preserves filtered colimits.

Let us consider a reflective full subcategory $\categ E$ of $\Psh(\categ D)$ that is stable, that contains the image of $\categ C$ and that is stable through limits and filtered colimits. We have an adjunction
$$
\begin{tikzcd}
\categ E \ar[rr, shift left, "L"]
&& \categ C \ar[ll, shift left,"R"]
\end{tikzcd}
$$
which factors the previous adjunction relating $\Psh(\categ D)$ and $\categ C$.

\begin{lemma}\label{lemmacoreflect}
The adjunction $L \dashv R$ is a coreflective localisation in the sense that the functor $L$ is fully faithful.
\end{lemma}

\begin{proof}
We can notice that $R$ preserves finite limits. Hence it preserves finite colimits (by stability). Since it preserves filtered colimits, then it preserves all colimits by \cite[Corollary 4.2.3.12]{Lurie17}.
Any element $d \in \Psh(\categ D)$ in the image of the Yoneda embedding $D \to \Psh(\categ D)$ belongs to the image of $\categ C$ and thus belongs to $\categ E$. For any such element, the map $d \to RL(d)$ is an equivalence by Yoneda's lemma (since $\categ D$ is a full subcategory of $\categ C$). Thus, since $R$ and $L$ preserves colimits and since $\categ E$ is generated by colimits of diagrams made up of such elements $d$, then for any $e \in \categ E$, the map $e \to RL(e)$ is an equivalence.
\end{proof}

\begin{proposition}\label{propositionstableequivalence}
The adjunction $L \dashv R$ relating $\categ C$ to $\categ E$ is an adjoint equivalence if and only if the right adjoint functor $R$ is conservative.
\end{proposition}

\begin{proof}
It is a direct consequence of Lemma \ref{lemmacoreflect}.
\end{proof}

\subsection{Strongly monadic functor}

\begin{definition}
A functor $F :\categ A \to \categ C$ is strongly monadic if $\categ A$ has all geometric realisations and if $F$ is right adjoint and conservative and preserves geometric realisations. Moreover, an adjunction is strongly monadic of its right adjoint is strongly monadic.
\end{definition}

\begin{notation}
A strongly monadic functor from $\categ A$ to $\categ C$ will be denoted $U^M$, its left adjoint will be denoted $T_M$ and $M$ will be the related monad $M = U^M \circ T_M$.
\end{notation}

\begin{lemma}\cite{RiehlVerity16}\label{lemmamonadic}
Let $U^M : \categ A \to \categ C$ be a (strongly) monadic functor, and let $Q = T_M U^M$. Then one has
one gets a functor
\begin{align*}
    (\Delta^\op)^{\triangleright} = \Delta_+^\op &\to \Fun{\categ A}{\categ A}
    \\
    n &\mapsto Q^{n+1} ,
\end{align*}
so that for any object $A \in \categ A$, the augmented simplicial object $n \mapsto Q^{n+1} A$ is the colimit of the underlying simplicial diagram.
\end{lemma}

\begin{proof}
Let $\categ{Adj}$ be the 2-category that encodes adjunctions. It has two objects $0$ and $1$ and is generated by morphisms $l : 0 \to 1$ and $r : 1 \to 0$ and 2-morphisms $\eta : \id_0 \to rl$ and $\epsilon : lr \to \id_1$ that satisfy the same relations than those defining adjunctions. We know from \cite{RiehlVerity16}, that the adjunction $T_M \dashv U^M$ induces a functor of categories enriched in $\infty$-categories
$$
\categ{Adj} \to \categ{QuasiCats} .
$$
that maps $l$ to the left adjoint and $r$ to the right adjoint. This, gives us the functor
\begin{align*}
\Delta_+^\op \simeq \categ{Adj}(1,1) &\to \Fun{\categ A}{\categ A} ,
\end{align*}
that sends $n \in \Delta_+^\op$ to $Q^{n+1}$.
Besides, composition with $r$ in the 2-category $\categ{Adj}$ recovers the canonical inclusion
$$
\Delta_+^\op \simeq \categ{Adj}(1,1) \xrightarrow{ r \circ - } \categ{Adj}(1,0) \simeq \Delta_\infty^\op .
$$
Subsequently, the image through $U^M \circ -$ of the augmented simplicial object $n \mapsto Q^{n+1}$ splits. In particular, for any object $A \in \categ A$ the augmented simplicial object in 
$\categ A$
$$
n \mapsto Q^{n+1} A
$$
is $U^M$-split. Hence, by monadicity it is colimiting.
\end{proof}

\begin{lemma}\label{lemmakanextension}
Let $D : \categ I \times \categ J \to \categ E$ be a functor of $\infty$-categories. The two following assertions are equivalent
\begin{enumerate}
    \item $D$ admits a left Kan extension with respect to the inclusion
$$
\categ I \times \categ J \to \categ I^\triangleright \times \categ J;
$$
\item the diagram
$$
D_j : \categ I \simeq \categ I \times \{j\} \hookrightarrow
\categ I \times \categ J \to \categ E
$$
admits a colimit for any $j \in \categ J$.
\end{enumerate}
Moreover, in the case these are true, the restriction of the left Kan extension to $\categ I^\triangleright \times \{j\} $ is the colimit of $D_j$ for any $j \in \categ J$.
\end{lemma}

\begin{proof}
Let us denote $\ast$ the final object of $\categ I^\triangleright$.
From \cite[Lemma 4.3.2.13]{Lurie17}, we know that the assertion (1) is equivalent to the fact that the functor
$$
(\categ I \times \categ J)_{/(\ast, j)} \to \categ I \times \categ J \to \categ E
$$
admits a colimit for any $j \in \categ J$. Then this colimit will give the value of the left Kan extension of $D$ on the object $(\ast, j)$.

To conclude, it suffices to notice that the functor
$$
\categ I \simeq \categ I \times \{j\} \to
(\categ I \times \categ J)_{/(\ast, j)}
$$
is cofinal.
\end{proof}

\begin{lemma}\label{lemmafunctoralkan}\cite[Corollary 4.3.2.16 and Proposition 4.3.2.17]{Lurie17}
Let $i : \categ K \to \categ J$ be an inclusion functor. If every functor $F$ from $\categ I$ to an $\infty$-category $\categ E$ admits a left Kan extension, then the functor
$$
\Fun{\categ J}{\categ E} \to
\Fun{\categ K}{\categ E}
$$
given by precomposition with $i$ admits a left adjoint whose value on any functor is given by its left Kan extension.
\end{lemma}

\begin{proposition}\label{propositionstronglymonadic}
Let $U^M : \categ A \to \categ C$ be a strongly monadic functor. If the $\infty$-category $\categ C$ is cocomplete, then $\categ A$ is also cocomplete.
\end{proposition}

\begin{proof}
Let us denote $Q = T_M \circ U^M$. Let us consider a diagram functor $D : K \to \categ A$. Let us show that $D$ admits a colimit. From Lemma \ref{lemmamonadic} we get a functor
\begin{align*}
   D^{\Delta_+} :  K \times \Delta_+^\op &\to \categ A
    \\
    (k,n) &\mapsto Q^{n+1} D(k) .
\end{align*}
We know from Lemma \ref{lemmatechnicalmonadic}, that the functor $D^{\Delta_+}$
admits a left kan extension $X$ with respect to the inclusion
$$
K \times \Delta_+^\op \to K^\triangleright \times \Delta_+^\op .
$$
Then by Lemma \ref{lemmakanextension}, the functor
$$
K \simeq K \times \{-1\} \hookrightarrow K \times \Delta_+^\op \to \categ A
$$
which is just $D$ admits a colimit.
\end{proof}

\begin{lemma}\label{lemmatechnicalmonadic}
The functor $D^{\Delta_+}$ described in the proof of Proposition \ref{propositionstronglymonadic}
admits a left kan extension $X$ with respect to the inclusion
$$
K \times \Delta_+^\op \to K^\triangleright \times \Delta_+^\op .
$$
\end{lemma}

\begin{proof}
Let us consider the following commutative square
$$
\begin{tikzcd}
K \times \Delta^\op
\ar[r,"i"] \ar[d, "j"]
& K^\triangleright \times \Delta^\op
\ar[d,"k"]
\\
K \times \Delta_+^\op
\ar[r,"l"]
& K^\triangleright \times \Delta_+^\op .
\end{tikzcd}
$$
Let $D^\Delta$ be the restriction of $D^{\Delta_+}$ onto $K \times \Delta^\op$. Then, by Lemma \ref{lemmakanextension} and Lemma \ref{lemmamonadic}, $D^{\Delta_+}$ is the left Kan extension of $D^\Delta$
along $j$.

Let us show that $D^\Delta$ admits a left Kan extension with respect to $k \circ i \simeq l \circ j$. First, for any $n \in \Delta$, the functor
$$
K \simeq K \times \{n\} \hookrightarrow K \times \Delta^\op \xrightarrow{D^\Delta} \categ A
$$
admits a colimit since it decomposes as
$$
K \xrightarrow{U^M Q^n D} \categ C \xrightarrow{T_M} \categ A
$$
and since $\categ C$ is cocomplete and $T_M$ preserves colimits. Hence by Lemma \ref{lemmakanextension}, $D^\Delta$ admits a left Kan extension along $i$. Again by Lemma \ref{lemmakanextension} and since $\categ A$ admits geometric realisations, then this left Kan extension admits a further left Kan extension along $k$, that we denote $X$. Then, $X$ is the left Kan extension of $D^\Delta$ along $k \circ i \simeq l \circ j$.

Using Lemma \ref{lemmakanextension}, we can prove that the restriction of $X$ to $K \times \Delta_+^\op$ is a left Kan extension of $D^\Delta$ with respect $j$, as is $D^{\Delta_+}$. Hence, both functors are equivalent (apply for instance Lemma \ref{lemmafunctoralkan}). Moreover, $X$ is a left Kan extension of $D^{\Delta_+}$ (\cite[Proposition 4.3.2.8]{Lurie17}).
\end{proof}

\begin{definition}
An adjunction is called $\omega$-accessible if the source of its right adjoint has filtered colimits and if this right adjoint preserves filtered colimits.
\end{definition}

\begin{corollary}
Let $U^M : \categ A \to \categ C$ be a strongly monadic functor and let us suppose that $\categ C$ is cocomplete. Then, the two following assertions are equivalent.
\begin{enumerate}
    \item the adjunction $T_M \dashv U^M$ is $\omega$-accessible, that is $U^M$ preserve filtered colimits;
    \item $U^M$ preserves small sifted limits.
\end{enumerate}
\end{corollary}

\begin{proof}
Since $\categ C$ is cocomplete, then by Proposition \ref{propositionstronglymonadic}, $\categ A$ is also cocomplete. Then, the result is a direct application of the fact that a functor from a cocomplete category preserves sifted colimits if and only if it preserves filtered colimits and geometric realisations (\cite[Corollary 5.5.8.17]{Lurie17}).
\end{proof}

\subsection{Presentations of algebras}
\subsubsection{The context}
\label{thecontext}
Let $\categ C$ be an $\omega$-presentable $\infty$-category. Thus, let us consider a small $\infty$-category $\categ D$, a functor $i: \categ D \to \categ C$ and the induced adjunction
$$
\begin{tikzcd}
\Psh(\categ D)
\ar[rr, shift left, "i_!"]
&&
\categ C
\ar[ll, shift left, "i^\ast"]
\end{tikzcd}
$$
so that the map $i_! i^\ast \to \id$ is an equivalence and so that $i^\ast$ preserves filtered colimits.

Besides, let us consider another adjunction
$$
\begin{tikzcd}
\categ C
\ar[rr, shift left, "T_M"]
&&
\categ A
\ar[ll, shift left, "U^M"]
\end{tikzcd}
$$
that is $\omega$-accessible and strongly monadic. Thus, $\categ A$ is cocomplete by Proposition \ref{propositionstronglymonadic} and $U^M$ preserves sifted colimits.

Let us consider a full subcategory embedding $j : \categ B \hookrightarrow \categ A$ so that $\categ B$ is small and so that it contains the image of $\categ D$, in the sense that for any object $d$ of $\categ D$ its image through the functor $T_M \circ i :  \categ D \to \categ A$ belong essentially to $\categ B$. We thus get a functor $t : \categ D \to \categ B$ so that the following square of $\infty$-categories
is commutative (up to homotopy).
$$
\begin{tikzcd}
\categ B
\ar[d,swap, "j"]
& \categ D
\ar[l,"t"] \ar[d, "i"]
\\
\categ A
& \categ C
\ar[l,"T_M"]
\end{tikzcd}
$$
From the universal property of preasheaves categories and the fact that $T_M$ preserves colimits, one gets another commutative square
$$
\begin{tikzcd}
\categ{Psh}(\categ B)
\ar[d,swap,"j_!"]
& \categ{Psh}(\categ D)
\ar[l,"t_!"] \ar[d,"i_!"]
\\
\categ A
& \categ C .
\ar[l,"T_M"]
\end{tikzcd}
$$
Noticing that all the functors in the square just above are left adjoint, one gets another square
$$
\begin{tikzcd}
\categ{Psh}(\categ B)
\ar[r,"t^\ast"]
& \categ{Psh}(\categ D)
\\
\categ A
\ar[r,"U^M"] \ar[u,"j^\ast"]
& \categ C .
\ar[u,"i^\ast"]
\end{tikzcd}
$$
whose functors are all right adjoints. We can notice that the right adjoint functor 
\begin{align*}
t^\ast : \Psh(\categ D) &\to \Psh(\categ B)    
\\
F & \mapsto F \circ t
\end{align*}
preserves colimits. Actually, it has a further right adjoint given by right Kan extension $t_\ast$ of $t$.

\begin{hypothesis}
Let us suppose that: 
\begin{itemize}
    \itemt the full subcategory $\categ B$ is stable under finite coproducts;
    \itemt its objects are $\omega$-small in $\categ A$; equivalently, the functor $\categ A \to \Psh(\categ B)$ preserves filtered colimits.
\end{itemize}
\end{hypothesis}

In particular, the image of the functor $\categ A \to \categ{Psh}(\categ B)$ lies in $\categ{Psh}_{\Sigma}(\categ B)$.

Moreover, let $\categ E$ be a full subcategory of $\categ{Psh}_{\Sigma}(\categ B)$ that satisfies the following hypothesis.

\begin{hypothesis}
We suppose that
\begin{itemize}
    \itemt the subcategory $\categ E$ contains all the images of objects of $\categ A$;
    \itemt the functor $\categ E \to  \categ{Psh}_{\Sigma}(\categ B)$ has a left adjoint and preserves filtered colimits; equivalently, the functor $\categ E \to  \categ{Psh}(\categ B)$ has a left adjoint and preserves filtered colimits.
\end{itemize}
\end{hypothesis}

The adjunction $j_! \dashv j^\ast$ relating $\categ A$ to $\Psh(\categ B)$ decomposes into a sequence of two adjunctions
$$
\begin{tikzcd}
\categ{Psh}(\categ B) \ar[rr, shift left,"{p}"]
&& \categ{E} \ar[rr, shift left,"{l}"] \ar[ll, shift left, "q"]
&& \categ{A}  \ar[ll, shift left,"r"]
\end{tikzcd}
$$
where the functor $l$ is just given by $j_! \circ q$.

\begin{remark}
The hypothesis are redundant. For instance the fact that the functor $U_M$ preserves filtered colimits follows from the fact that the composite functor $\categ A \to \Psh(\categ B) \to \Psh(\categ D)$ preserves filtered colimits and the functor $\categ C \to \Psh(\categ D)$ preserves filtered colimits and is conservative.
\end{remark}

\subsubsection{Presenting algebras}

Our goal is to prove the following theorem.

\begin{theorem}\label{thmmain}
The following assertions are equivalent.
\begin{enumerate}
    \item the adjunction $l \dashv r$ relating $\categ A$ to $\categ E$ is an adjoint equivalence;
    \item the composite functor
    $$
\categ E \xrightarrow{q} \Psh(\categ B) \xrightarrow{t^\ast} \Psh(\categ D)
$$
is conservative and sends elements of $\categ E$ to the essential image of the functor $i^\ast$ from $\categ C$ to $\Psh(\categ D)$;
    \item the composite functor
$$
\categ E \xrightarrow{q} \Psh(\categ B) \xrightarrow{t^\ast} \Psh(\categ D) \xrightarrow{i_!} \categ C
$$
is conservative.
\end{enumerate}
\end{theorem}

\begin{proof}
First, let us suppose (1), that is $l \dashv r$ is an adjoint equivalence, then
$$
t^\ast \circ q \simeq t^\ast \circ q \circ r \circ l
\simeq i^\ast \circ U^M
$$
is conservative and its image lies in the essential image of $i^\ast$. This proves (2).

Let us suppose (2). For any morphism $f:X \to Y$ in $\categ E$, if $i_! t^\ast q(f)$ is an equivalence, then $i\ast i_!t^\ast q(f)$ is an equivalence and so $t^\ast q(f)$ is an equivalence since the maps $X \to i^\ast i_! (X)$ and $Y \to i^\ast i_! (Y)$ are equivalences. Thus, $f$ is an equivalence since $t^\ast q$ is conservative. So $i_! t^\ast q$ is conservative, that is (3) is true.

Finally, let us suppose (3), that is $i_! \circ u^\ast \circ q$ is conservative. We can prove (1), that is the adjunction $l \dashv r$ is an adjoint equivalence by showing that
\begin{enumerate}
    \item the functor $r$ is conservative
    \item the map $\id_{\categ E} \to rl$ is an equivalence.
\end{enumerate}
The first point is Lemma \ref{lemmaconservative} and the second point is Lemma \ref{lemmaequivalence}
\end{proof}

\begin{lemma}\label{lemmaconservative}
If the composite functor $i_! \circ t^\ast \circ q$ is conservative then the functor $r$ is conservative.
\end{lemma}

\begin{proof}
One can notice that $U_M \simeq i_! \circ t^\ast \circ q \circ r$. Since $U_M$ and $i_! \circ t^\ast \circ q$ are conservative, then $r$ is also conservative. 
\end{proof}

\begin{lemma}\label{lemmaind}
The morphism $F \to j^\ast j_!(F)$ is an equivalence for any element $F \in \Ind_{\omega}(\categ B) \subseteq \categ{Psh}(\categ B)$.
\end{lemma}

\begin{proof}
This is an equivalence for any $F$ that belong to the image of the Yoneda embedding of $\categ B$ since $B$ is a full sucategory of $\categ A$. To conclude, it suffices to notice that both functors $j_!$ and $j^\ast$ preserves filtered colimits.
\end{proof}

\begin{lemma}\label{lemmaequivalence}
If the composite functor $i_! \circ t^\ast \circ q$ is conservative then the morphism $F \to j^\ast j_!(F)$ is an equivalence for any element $F \in \categ E \subseteq \categ{Psh}(\categ B)$.
\end{lemma}

\begin{proof}
Since $F$ belongs in particular to $\categ{Psh}_{\Sigma}(\categ B)$, it is the colimit in $\categ{Psh}(\categ B)$ of a simplicial object
$$
(G_n)_{n \in \Delta^{op}}
$$
where each $G_n$ belongs to $\Ind_{\omega}(\categ B)$.

Then, the morphism $F \to j^\ast j_!(F)$ decomposes as
$$
\begin{tikzcd}
 \varinjlim_{n \in \Delta^{op}} G_n
 \ar[r] \ar[d,"\simeq"]
 & \varinjlim_{n \in \Delta^{op}} j^\ast j_!(G_n)
 \ar[r]
 & j^\ast(\varinjlim_{n \in \Delta^{op}} j_!(G_n))
 \ar[r]
 & j^\ast j_!(\varinjlim_{n \in \Delta^{op}} G_n)
 \ar[d,"\simeq"]
 \\
 F
 \ar[rrr]
 &&& j^\ast j_! (F) .
\end{tikzcd}
$$
Let us show that the three following maps are equivalences:
\begin{enumerate}
    \item $\varinjlim_{n \in \Delta^{op}} G_n \to\varinjlim_{n \in \Delta^{op}} j^\ast j_!(G_n)$;
    \item $j^\ast(\varinjlim_{n \in \Delta^{op}} j_!(G_n))
\to j^\ast j_!(\varinjlim_{n \in \Delta^{op}} G_n)$;
    \item $\varinjlim_{n \in \Delta^{op}} j^\ast j_!(G_n)
 \to
 j^\ast(\varinjlim_{n \in \Delta^{op}} j_! (G_n))$.
\end{enumerate}
We can notice that since each $G_n$ is in $\Ind_{\omega}(\categ B)$, then the map $G_n \to j^\ast j_!(G_n)$ is an equivalence by Lemma \ref{lemmaind}. Thus, the map (1) is an equivalence. In particular, its target belongs to $\categ{E}$.
The map (2) is also an equivalence since $j_!$ is left adjoint.

It remains to show that the map (3)
 is an equivalence. Since this morphism belong to the essential image of $q$ (both the target and the source essentially belong to the full sucategory $\categ E$), it suffices to show that its image through $i_! \circ t^\ast$ is an equivalence (since $i_! \circ t^\ast \circ q$ is conservative). Let us consider the composite map
$$
\varinjlim_{n \in \Delta^{op}} i_!  t^\ast j^\ast j_!  (G_n)
\to 
i_!  t^\ast (\varinjlim_{n \in \Delta^{op}} j^\ast j_!(G_n))
\to i_!  t^\ast j^\ast( \varinjlim_{n \in \Delta^{op}} j_!(G_n))
$$
Its first part is an equivalence since $i_! \circ t^\ast$ preserves colimits. Besides, since $i_! \circ  t^\ast \circ j^\ast \simeq U_M$ preserves geometric realisations, then the composite map is also an equivalence. Hence, the second map
$$
i_!  t^\ast (\varinjlim_{n \in \Delta^{op}} j^\ast j_!(G_n))
\to i_!  t^\ast j^\ast( \varinjlim_{n \in \Delta^{op}} j_!(G_n))
$$
is also an equivalence. Hence, $F \to j^\ast j_! (F)$ is an equivalence.
\end{proof}


\section{Chain complexes and their algebras}

Let $R$ be a ring. The goal of this section is to describe two types of presentations of an $\infty$-categories of algebras in chain complexes.

We first recall some results about the link 
between combinatorial model categories and $\infty$-categories.
Then we describe a presentation of the $\infty$-category of chain complexes and of its product over a set $K$. This allows us to apply Theorem \ref{thmmain} to describe the cellular and the cartesian presentations of an $\infty$-category related to chain complexes through an $\omega$-accessible strongly monadic adjunction.
We conclude this section by describing how these presentations fit the context of algebras over a coloured operad in chain complexes.

\subsection{Combinatorial model categories}

Let $\categ C$ be a combinatorial model category and let $\categ C^c$ be its full subcategory of cofibrant objects.

\begin{notation}
For any marked simplicial $(X,\mathrm W)$, we will denote by $X[\mathrm{W}^{-1}]$ the underlying quasi-category of a fibrant replacement of $(X,\mathrm W)$ for the cartesian model structure on marked simplicial sets. This applies for instance to relative categories.
\end{notation}

\begin{lemma}\cite[Corollary 3.1.4.3.]{Lurie17}
The cartesian model structure on the category of marked simplicial sets is monoidal with respect to
the cartesian monoidal structure.
\end{lemma}

\begin{lemma}
The functor
$$
\categ C^c[\mathrm W^{-1}] \to \categ C[\mathrm W^{-1}]
$$
is an equivalence of $\infty$-categories.
\end{lemma}

\begin{proof}
The functor of relative categories from $(\categ C^c, \mathrm W)$ to $(\categ C, \mathrm W)$ induces a morphism of marked simplicial sets
$$
f : (N(\categ C^c), \mathrm W) \to (N(\categ C), \mathrm W) .
$$
It suffices to show that this morphism $f$ is an equivalence of marked simplicial sets.

Since the model category $\categ C$ is combinatorial, it has a cofibrant replacement functor which gives another morphism of marked simplicial sets
$$
g : (N(\categ C), \mathrm W) \to (N(\categ C^c), \mathrm W) .
$$
Let $e$ be the marked simplicial set which consists in a marked arrow. Then for any marked simplicial set $X$, $X \times e$ is a cylinder object for $X$. Moreover, the cofibrant replacement functor gives us maps
$$
(N(\categ C), \mathrm W) \times e \to (N(\categ C), \mathrm W); \quad (N(\categ C^c), \mathrm W) \times e
\to (N(\categ C^c), \mathrm W)
$$
which relate on the one hand $fg$ to $\id$ and on the other hand $gf$ to $\id$. Hence $f$ is an equivalence.
\end{proof}

\begin{proposition}\cite[Proposition 1.3.4.22]{Lurie18}\label{propositioncombi}
The $\infty$-category $\categ C^c [\mathrm W^{-1}]$ is $\kappa$-presentable for some regular cardinal $\kappa$.
\end{proposition}

\begin{proposition}\cite[Proposition 1.3.4.24]{Lurie18}\label{propositionhomotopycolim}
Let us consider a functor $D : \categ I \to \categ C^c$ where $\categ I$ is a small category. Then, a cocone of $D$ is its homotopy colimit if and only if the induced cocone in $\categ C^c[\mathrm{W}^{-1}]$ is colimiting.
\end{proposition}

\begin{proposition}\cite[Proposition 1.3.4.25]{Lurie18}\label{propositioncatfunccolim}
For any small category $\categ I$, the category $\Fun{\categ I}{\categ C}$ admits a combinatorial model structure whose weak equivalences are objectwise weak equivalence.
Moreover, the functor
$$
 \Fun{\categ I}{\categ C^c}[\mathrm W^{-1}]
 \to \Fun{\categ I}{{\categ C^c}[\mathrm W^{-1}]}
$$
is an equivalence of $\infty$-categories.
\end{proposition}

\begin{lemma}\label{lemmapreservefiltered}
Let $\kappa$ be a regular cardinal and let $F : \categ C \to \categ D$ be a functor between combinatorial model categories that preserves weak equivalences. Let us suppose weak equivalences in $\categ C$ and $\categ D$ are stable through $\kappa$-filtered colimits and that $F$ preserves $\kappa$-filtered colimits. Then the induced functor
$$
\categ C[\mathrm{W}^{-1}] \to \categ D[\mathrm{W}^{-1}]
$$
preserves $\kappa$-filtered colimits.
\end{lemma}

\begin{proof}
For any $\kappa$-filtered $\infty$-category $\categ I$, there exists a $\kappa$-filtered poset $\categ P$ and a cofinal map from this poset to $\categ I$ (\cite[Proposition 5.3.1.18]{Lurie17}). Thus, using Proposition \ref{propositionhomotopycolim} and Proposition \ref{propositioncatfunccolim}, it suffices to prove that $F$ preserves homotopy colimits indexed by $\kappa$-filtered posets. To conclude, such homotopy colimits are given by colimits in these two model categories.
\end{proof}

\begin{proposition}\label{propositionQuillenadj}\cite{MazelGee16}
Let us consider a Quillen adjunction
$$
\begin{tikzcd}
 \categ C
 \ar[rr, shift left,"L"]
 && \categ D \ar[ll, shift left,"R"]
\end{tikzcd}
$$
between combinatorial model categories. Let $Q$ be a cofibrant replacement functor for $\categ C$ and let $F$ be a fibrant replacement functor for $\categ D$. Then the functors derived from $LQ$ and $RF$ that relate $\categ C[\mathrm{W^{-1}}]$ and $\categ D[\mathrm{W^{-1}}]$ form a pair of a left and a right adjoint functors.
\end{proposition}

\begin{lemma}\label{lemmalocalisationproduct}
Let $(X_i)_{i \in I}$ be a collection of marked simplicial sets. Then the functor
$$
(\prod_{i \in I} X_i) [\mathrm{W}^{-1}] \to \prod_{i \in I} ( X_i[\mathrm{W}^{-1}])
$$
is an equivalence.
\end{lemma}

\begin{proof}
It amounts to show that the morphism of marked simplicial sets
$$
\prod_{i \in I} X_i \to
\prod_{i \in I} ( X_i[\mathrm{W}^{-1}])
$$
is a weak equivalence.
If $I$ is finite, this follows from the fact that the cartesian model structure of marked simplical sets is monoidal. In the general case, it suffices to notice that the transfinite composition of weak equivalences in marked simplicial sets is a weak equivalence.
\end{proof}

\subsection{Chain complexes}

\begin{definition}
Let $\categ{Chain}_R$ be the category of chain complexes of  $R$-modules. For any $m \in \mathbb Z$, we denote $S^m$ the chain complex which consists in $R$ in degree $m$ and zero on any other degree. Moreover, $D^m$ is the chain complex which consists in $R$ in degree $m$ and $m-1$ with differential given by the identity of $R$ and which is zero in any other degree.
\end{definition}

\begin{proposition}\cite{Hovey99}
The category $\categ{Chain}_R$ of chain complexes of  $R$-modules admits a combinatorial model structure whose
\begin{itemize}
    \itemt fibrations are degreewise surjections;
    \itemt weak equivalences are quasi-isomorphisms;
    \itemt generating cofibrations are given by maps
    $$
        S^n  \to D^{n+1}, \ n \in \mathbb Z;
    $$
    \itemt generating acyclic cofibrations are given by maps
    $$
    0 \to D^n, n \in \mathbb Z;
    $$
    \itemt cofibrations are degreewise injections whose cokernel is degreewise projective.
\end{itemize}
The resulting $\infty$-category of chain complexes $\categ{Chain}_R[\mathrm W^{-1}]$ is denoted $\categ{Ch}_R$.
\end{proposition}

\begin{proposition}\cite[Proposition 1.3.4.5]{Lurie18}
Let $\categ{Chain}^{\categ{dg}}_R$ be the differential graded category of chain complexes of $R$-modules and let $\categ{Chain}^{\categ{dg},c}_R$ be the full subcategory spanned by objects that are cofibrant for the model structure described above.
The functor
$$
\categ{Chain}^{c}_R \to \mathrm{N}_{\mathrm{dg}}(\categ{Chain}^{\categ{dg},c}_R)
$$
induces an equivalence of $\infty$-categories
$$
\categ{Ch}_R \simeq \categ{Chain}^{c}_R[\mathrm{W}^{-1}] \to \mathrm{N}_{\mathrm{dg}}(\categ{Chain}^{\categ{dg},c}_R) .
$$
\end{proposition}

\begin{proposition}\cite[Proposition 1.3.2.10]{Lurie18}
The $\infty$-category $\categ{Ch}_R$ is stable.
\end{proposition}

\begin{definition}
Let $n \in \mathbb Z \sqcup \{\infty\}$.
We denote $\categ{Loops}_{R,\leq n}$ the full subcategory of the $\infty$-category of chain complexes $\categ{Ch}_{R}$ spanned by the objects 
$$
0, S^{m}, m \leq n.
$$
\end{definition}

The $\infty$-category $\categ{Loops}_{R,\leq n}$ is pointed with $0$ as zero object. Moreover, for any
$m \leq n$
$$
S^m = \Sigma S^{m-1}; `\quad \Omega S^m = S^{m-1}.
$$

\begin{definition}
Let $\Psh_{\categ{stab}}(\categ{Loops}_{R, \leq n})$ be the full subcategory of $\Psh(\categ{Loops}_{R, \leq n})$ spanned by the functors $F : \categ{Loops}_{R, \leq n}^\op \to \categ S$ that sends $0$ to the final $\infty$-groupoid $\ast$ and so that the map
$$
F(S^{m}) \simeq F(\Sigma \Omega S^{m}) \to \Omega  (F(\Omega S^m ) )
$$
is an equivalence for any $m \leq n$. A functor that belongs to $\Psh_{\categ{stab}}(\categ{Loops}_{R, \leq n})$ will be called a stable functors.
\end{definition}

\begin{lemma}
The $\infty$-category $\Psh_{\categ{stab}}(\categ{Loops}_{R, \leq n})$
is stable.
\end{lemma}

\begin{proof}
From Lemma \ref{lemmapresentatio}, it is $\omega$-presentable hence complete.
The final object of $\Psh(\categ{Loops}_{R, \leq n})$ is the image through the Yoneda embedding of $0 \in \categ{Loops}_{R, \leq n}$. It belongs to the full subcategory of stable functors and then is also the final object of this subcategory. Moreover, for any stable functor $F$, we have equivalences
$$
\Mapp{\Psh(\categ{Loops}_{R, \leq n})}{\ast}{F}
\simeq
F(0) \simeq \ast .
$$
Hence, this final object is a zero object.  Moreover, the loop functor is an equivalence whose pseudo inverse is the precomposition with the loop functor of
$\categ{Loops}_{R, \leq n}$.
\end{proof}

\begin{lemma}
The inclusion functor
$$
\Psh_{\categ{stab}}(\categ{Loops}_{R, \leq n}) \hookrightarrow \Psh(\categ{Loops}_{R, \leq n})
$$
is right adjoint and preserves filtered colimits.
\end{lemma}

\begin{proof}
This is just an application of Lemma \ref{lemmapresentatio}.
\end{proof}

\begin{lemma}
Any object of the form $S^m$ for $m \in \mathbb Z$ is $\omega$-small in $\categ{Ch}_R$. 
\end{lemma}

\begin{proof}
First, the equivalence $\Sigma$ and thus its inverse $\Omega$ preserves $\omega$-small objects since $0$ is $\omega$-small and finite colimits of $\omega$-small objects are $\omega$-small. Hence, it suffices to prove that $S^0$ is small.

The Dold-Kan correspondence gives a Quillen adjunction relating the Kan Quillen model category of simplicial sets to the model category $\categ{Chain}_R$. Proving that $S^0 \in \categ{Ch}_R$ is $\omega$-small amounts to prove that the right derived functor
$$
\categ{Chain}[\mathrm W^{-1}] \to \categ{sSet}[\mathrm W^{-1}]
$$
preserves filtered colimits. This follows from Lemma \ref{lemmapreservefiltered}.
\end{proof}

\begin{proposition}\label{propositionchainpresentation}
Let $n \in \mathbb Z \sqcup \{\infty\}$.
The functor
$$
\categ{Ch}_R \to \Psh_{\categ{stab}}(\categ{Loops}_{R, \leq n})
$$
induced by the inclusion of
$\categ{Loops}_{R, \leq mn}$ into $\categ{Ch}_R$ is an equivalence.
\end{proposition}

\begin{proof}
We have a sequence of adjunctions
$$
\begin{tikzcd}
\Psh(\categ{Loops}_{R, \leq n})
\ar[rr, shift left]
&&
\Psh_{\categ{stab}}(\categ{Loops}_{R, \leq n}) \ar[rr, shift left]
\ar[ll, shift left]
&& \categ{Ch}_R .
\ar[ll, shift left]
\end{tikzcd}
$$
whose right adjoint functors preserve filtered colimits, since the objects of $\categ{Loops}_{R, \leq n}$ are all $\omega$-small in $\categ{Ch}_R$.
Moreover, the $\infty$-categories $\Psh_{\categ{stab}}(\categ{Loops}_{R, \leq n})$ and $\categ{Ch}_R$ are both stable. To conclude using Proposition \ref{propositionstableequivalence}, it suffices to notice that the functors
$$
X \in \categ{Ch}_R \mapsto \Mapp{\categ{Ch}_R}{S^m}{X} \in \categ S, \ m \leq n
$$
are jointly conservative.
\end{proof}

\subsection{Products of chain complexes}

Let $K$ be a set. We will call the elements of $K$ colours.

\begin{definition}
We denote
$$
\categ{Chain}_R^K = \Fun{K}{\categ{Chain}_R}; \quad \categ{Ch}_R^K = \Fun{K}{\categ{Ch}_R}.
$$
Moreover, for any $k \in K$, and any $m \in \mathbb Z$, let us denote $S^m_k$ the object of $\categ{Chain}_R^K$ given by the functor
that sends $k$ to $S^m$ and any other $k' \neq k$ to $0$. We will denote also $S^m_k$ its image in the $\infty$-category $\categ{Ch}_R^K$. Finally, we define similarly $D^m_k$.
\end{definition}

\begin{proposition}
The product category $\categ{Chain}_R^K$ admits a combinatorial model structure whose
\begin{itemize}
    \itemt fibrations are colourwise degreewise surjections;
    \itemt weak equivalences are colourwise quasi-isomorphisms;
    \itemt generating cofibrations are given by maps
    $$
        S_k^n  \to D_k^{n+1}, \ n \in \mathbb Z, k \in K;
    $$
    \itemt generating acyclic cofibrations are given maps
    $$
    0 \to D_k^n, n \in \mathbb Z;
    $$
    \itemt cofibrations are colourwise degreewise injections whose cokernel is colourwise degreewise projective.
\end{itemize}
\end{proposition}

\begin{proof}
This proposition is just an incarnation of the fact that the product of model categories inherits a model structure whose fibrations, cofibrations and weak equivalences are componentwise respectively fibrations, cofibrations and weak equivalences.
\end{proof}

\begin{proposition}
The functor
$$
\categ{Chain}_R^K \to \categ{Ch}_R^K
$$
induces an equivalence
$$
\categ{Chain}_R^K[\mathrm{W}^{-1}] \to \categ{Ch}_R^K .
$$
\end{proposition}

\begin{proof}
This is just an application of Lemma \ref{lemmalocalisationproduct}.
\end{proof}

\begin{definition}
For any $n \in \mathbb Z \sqcup \{\infty\}$,
let $\categ{Loops}_{R, \leq n}^{(K)}$ be the full subcategory of $\categ{Chain}_R^K[\mathrm{W}^{-1}]$ spanned by the objects, $0$ and $S^m_k$ for $m \leq n$ and $k \in K$.
\end{definition}

\begin{definition}
For any $n \in \mathbb Z \sqcup \{\infty\}$, let $\Psh_{\categ{stab}}(\categ{Loops}_{R, \leq n}^{(K)})$ the full subcategory of $\Psh(\categ{Loops}_{R, \leq n}^{(K)})$ spanned by functors $F$ so that
\begin{itemize}
    \itemt $F(0) \simeq \ast$;
    \itemt the map
    $$
    F(S^m_k) \simeq F(\Sigma \Omega S^{m-1}_k) \to  \Omega F(\Omega S^{m-1}_k)
    $$
    is an equivalence for any $m \leq n$ and any $k \in K$.
\end{itemize}
A functor that satisfies such conditions is called stable.
\end{definition}

\begin{proposition}\label{propositionchainproductpresentation}
For any $n \in \mathbb Z \sqcup \{\infty\}$,
the inclusion
$$
\categ{Loops}_{R, \leq n}^{(K)} 
\to
\categ{Chain}_R^K \left[\mathrm{W}^{-1}\right]
$$
induces and equivalence of $\infty$-categories between $\categ{Chain}_R^K[\mathrm{W}^{-1}]$ and the full subcategory $\Psh_{\categ{stab}}(\categ{Loops}_{R, \leq n}^{(K)})$ of $\Psh(\categ{Loops}_{R, \leq n}^{(K)})$.
\end{proposition}

\begin{proof}
This follows from the same arguments as in the proof of Proposition \ref{propositionchainpresentation}. Moreover, the ingredients of this proof (stability, smallness of objects $S^m_k$) are straightforward consequences of their counterparts in the case $K=\ast$.
\end{proof}

\begin{remark}
One can also prove that one has the following equivalence of $\infty$-categories
$$
\categ{Loops}_{R, \leq n}^{(K)} \simeq \categ{Loops}_{R, \leq n} \coprod_K \ast .
$$
This leads then to the fact that the functor
$$
Psh_{\categ{stab}}(\categ{Loops}_{R, \leq n}^{(K)})
\to \prod_{k \in K} Psh_{\categ{stab}}(\categ{Loops}_{R, \leq n})
$$
is an equivalence.
\end{remark}

\subsection{Algebras in chain complexes}

Let $K$ be a set called the set of colours. Let us consider an adjunction of $\infty$-categories
$$
\begin{tikzcd}
\categ{Ch}_{R}^K \ar[rr, shift left, "T_M"]
&& \categ{A} . \ar[ll, shift left, "U^M"]
\end{tikzcd}
$$
that is is strongly monadic and $\omega$-accessible. Our goal is to describe several presentations of the $\infty$-category $\categ A$

\subsubsection{The cellular presentation}

Let $n \in \mathbb Z \sqcup \{\infty\}$.

\begin{definition}
A $n$-elementary morphism $f : A \to B$ in $\categ A$ is a morphism so that there exists an integer $m \leq n-1$, a colour $k \in K$ and a pushout
$$
\begin{tikzcd}
 T_M(S_k^m) 
 \ar[r] \ar[d]
 & A
 \ar[d]
 \\
 T_M(0)
 \ar[r]
 & B .
\end{tikzcd}
$$
\end{definition}

\begin{definition}
A finitely $n$-cellular morphism $f : A \to B$ in $\categ A$ is the composition of a finite collection of $n$-elementary morphisms.
\end{definition}

\begin{definition}
An object $A \in \categ A$ is finitely $n$-cellular if the morphism $\emptyset = T_M(0) \to A$
is finitely $n$-cellular.
We denote $\categ{Cell}_{M,\leq n}$ the full subcategory of $\categ A$ spanned by finitely $n$-cellular objects.
\end{definition}

\begin{lemma}\label{lemmacellular}
The full subcategory $\categ{Cell}_{M,\leq n}$ of $\categ A$ is stable through coproducts and its objects are $\omega$-small in $\categ A$.
\end{lemma}

\begin{proof}
Let us first show that $\categ{Cell}_{M,\leq n}$ is stable through coproducts. Let $A$ and $B$ be finitely $n$-cellular objects and let us consider a sequence of $n$-elementary morphisms
$$
\emptyset = A_0 \xrightarrow{f_0} A_1 
\xrightarrow{f_1} \cdots \xrightarrow{f_{l-1}} A_l=A . 
$$
Then, let us consider the following sequence of morphisms from $B$ to $A \sqcup B$
$$
B = A_0 \sqcup B \xrightarrow{f_0 \sqcup \id} A_1 \sqcup B
\xrightarrow{f_1 \sqcup \id} \cdots \xrightarrow{f_{l-1} \sqcup \id} A_l \sqcup B =A \sqcup B . 
$$
Any of these maps $f_i \sqcup \id$ is $n$-elementary as we have two pushouts that yield another pushout as follows
$$
\begin{tikzcd}
 T_M(S^m_k) 
 \ar[r] \ar[d]
 & A_i
 \ar[r] \ar[d]
 & A_i \sqcup B
 \ar[d]
 \\
 T_M(0)
 \ar[r]
 & A_{i+1}
 \ar[r]
 & A_{i+1} \sqcup B .
\end{tikzcd}
$$

Let us show now that finitely $n$-cellular objects are $\omega$-small. Since the functor $U^M$ preserves filtered colimits, then its adjoint $T_M$ sends $\omega$-small objects to $\omega$-small objects. To conclude, finite colimits of $\omega$-small objects are $\omega$-small.
\end{proof}

\begin{theorem}\label{theoremcellularpresentation}
The $\infty$-category $\categ A$ is canonically equivalent to the full subcategory of $\Psh(\categ{Cell}_{M,\leq n})$ spanned by functors $F$ so that
\begin{enumerate}
    \item $F(\emptyset) \simeq \ast$;
    \item for any pushout square in $\categ C$ of the form
    $$
\begin{tikzcd}
 T_M(S^m_k) 
 \ar[r] \ar[d]
 & A
 \ar[d]
 \\
 T_M(0)
 \ar[r]
 & B
\end{tikzcd}
$$
its image through $F$ is a pullback.
\end{enumerate}
\end{theorem}

\begin{proof}
The following diagram
$$
\begin{tikzcd}
 \categ{Cell}_{M,\leq n}
 \ar[r] \ar[d]
 & \categ{Loops}_{R,\leq n}
 \ar[d]
 \\
 \categ A
 \ar[r]
 & \categ{Ch}_R^K .
\end{tikzcd}
$$
is commutative.
To check that we are in the context of Theorem \ref{thmmain}, we need to check that the full subcategory $\categ{Cell}_{M,\leq n}$ of $\categ A$ is stable through coproducts and that its objects are $\omega$-small in $\categ A$; this follows from Lemma \ref{lemmacellular}. We need moreover to check that the full subcategory of $\Psh(\categ{Cell}_{M,\leq n})$ described in the theorem
is reflective and is stable through filtered colimits; this is given by 
Lemma \ref{lemmapresentatio}.
We notice that this full subcategory contains the image of $\categ A$. Finally, using arguments similar as those used to prove Lemma \ref{lemmacellular}, one can check that any functor in the full subcategory of $\Psh(\categ{Cell}_{M,\leq n})$ described in the theorem preserves finite products.

Thus it suffices to check that condition (2) of Theorem \ref{thmmain} is satisfied, which is straightforward.
\end{proof}

\subsubsection{The cartesian presentation}

Let $n \in \mathbb Z \sqcup \{\infty\}$.

\begin{definition}
Let $\categ{Cart}_{M, \leq n}$ be the smallest full subcategory of $\categ{A}$ that contains $T_M{S_k^m}$ for any $m \leq n$, and $k \in K$ and that is stable through finite coproducts.
\end{definition}

\begin{theorem}\label{theoremcartesianpresentation}
The $\infty$-category $\categ A$ is canonically equivalent to the full subcategory of $\Psh(\categ{Cart}_{M, \leq n})$ spanned by functors $F$ so that
\begin{enumerate}
    \item $F$ preserves finite products;
    \item for any pushout square in $\categ{Cart}_{M, \leq n}$ of the form
    $$
\begin{tikzcd}
 T_M(X) 
 \ar[r] \ar[d]
 & T_M(0)
 \ar[d]
 \\
 T_M(0)
 \ar[r]
 & T_M(\Sigma X)
\end{tikzcd}
$$
its image through $F$ is a pullback.
\end{enumerate}
\end{theorem}

\begin{proof}
It follows from the same arguments as those used to prove Theorem \ref{theoremcellularpresentation}.
\end{proof}

\subsection{Algebras over an operad}

Let $\categ P$ be a coloured operad enriched in chain complexes and let $K$ be its set of colours. We have an adjunction relating the product over $K$ of the category of chain complexes to the category of $\categ P$-algebras
$$
\begin{tikzcd}
\categ{Chain}_R^K \ar[rr, shift left, "T_P"]
&& \catofalg{\categ P} . \ar[ll, shift left, "U^P"]
\end{tikzcd}
$$

\begin{definition}
The operad $\operad P$ is called admissible if there exists a combinatorial model structure on the category of $\categ P$-algebras transferred  from that of $\categ{Chain}_R^K$, that is
\begin{itemize}
    \itemt the weak equivalences are the morphisms that are colourwise quasi-isomorphisms;
    \itemt the fibrations are the morphisms that are colourwise and degreewise surjections.
\end{itemize}
Then, the adjunction $T_P \dashv U^P$ becomes a Quillen adjunction.
\end{definition}

\begin{definition}
The operad $\operad P$ is $\Sigma$-cofibrant if the underlying $K$-coloured symmetric sequence is cofibrant in the projective model structure on $K$-coloured symmetric sequences in chain complexes.
\end{definition}

\begin{remark}
If $R$ is rational or if $\operad P$ is planar, then $\operad P$ is admissible (see for instance \cite{PavlovScholbach} and \cite{BergerMoerdijk07}).
Moreover, in \cite{Hinich97}, Hinich proves that a monochromatic $\Sigma$-split operad is admissible. However, to the best of my knowledge, the case of $\Sigma$-cofibrant operads in chain complexes is a blindspot in the literature about admissible operads.
\end{remark}

\begin{proposition}\label{propositionconservative}
Let us suppose that $\operad P$ is admissible. Then, the functor
$$
\catofalg{\categ P} [\mathrm{W}^{-1}] \to \categ{Chain}_R^K [\mathrm{W}^{-1}]\simeq  \categ{Ch}_R^K
$$
induced by $U^P$ is conservative.
\end{proposition}

\begin{proof}
Let us consider the following diagram of $\infty$-categories
$$
\begin{tikzcd}
\catofalg{\categ P} [\mathrm{W}^{-1}]
\ar[r] \ar[d]
& \categ{Chain}_R^K [\mathrm{W}^{-1}]
\ar[d]
\\
\mathrm{Ho}(\catofalg{\categ P})
\ar[r]
& \mathrm{Ho}(\categ{Chain}_R^K) ,
\end{tikzcd}
$$
where the notation $\mathrm{Ho}$ stands for the homotopy category of a model category.
The two vertical functors are conservative. Hence proving that the top horizontal functor is conservative amounts to prove that the bottom horizontal functor is conservative.

Let $f: A \to B$ be a morphism in $\mathrm{Ho}(\catofalg{\categ P})$ that is sent to an isomorphism in $\mathrm{Ho}(\categ{Chain}_R^K)$. We can suppose that $A$ is cofibrant and that $B$ is fibrant. Then $f$ proceeds from a morphism $f' : A \to B$ in $\catofalg{\categ P}$. Then,
$U^M (f')$ becomes an isomorphism in $\mathrm{Ho}(\categ{Chain}_R^K)$. Thus, general properties of model categories ensures us that $U^M (f')$ is a weak equivalence. Hence $f'$ is a weak equivalence. So $f$ is an isomorphism. Thus the right derived functor from $\mathrm{Ho}(\catofalg{\categ P})$ to $\mathrm{Ho}(\categ{Chain}_R^K)$ is conservative.
\end{proof}

\begin{proposition}\cite[Appendix A]{HarpazNuitenPrasma}\label{propositionsifted}
Let us suppose that $\operad P$ is $\Sigma$-cofibrant and admissible. Then, the functor
$$
\catofalg{\categ P}[\mathrm W^{-1}] \to \categ{Ch}_R^K
$$
induced by $U^P$
preserves sifted colimits.
\end{proposition}

\begin{corollary}
In the case where $\categ P$ is $\Sigma$-cofibrant and admissible, then the adjunction
$$
\begin{tikzcd}
\categ{Ch}^K_R \ar[rr, shift left]
&& \catofalg{\categ P}[\mathrm{W^{-1}}]
\ar[ll, shift left]
\end{tikzcd}
$$
induced by the Quillen adjunction $T_P \dashv U^P$ is strongly monadic and $\omega$-accessible.
\end{corollary}

\begin{proof}
This follows from Proposition \ref{propositionconservative} and Proposition \ref{propositionsifted}
\end{proof}

Thus, in that case where $\categ P$ is $\Sigma$-cofibrant and admissible, one can apply Theorem \ref{theoremcellularpresentation} or Theorem \ref{theoremcartesianpresentation} to get presentations of the $\infty$-category $\catofalg{\categ P}[\mathrm{W}^{-1}]$.
One can notice that for any $n \in \mathbb Z \sqcup \{\infty\}$:
\begin{enumerate}
    \item an object $A \in \catofalg{\categ P}[\mathrm{W}^{-1}]$ is finitely $n$-cellular if and only if it is equivalent to the image of an element $A'$ in the category $\catofalg{\categ P}$ so that there exists a finite sequence of morphisms in this category
    $$
    A_0 \xrightarrow{f_0} A_1 \xrightarrow{f_1} \cdots \xrightarrow{f_{l-1}} A_l = A'
    $$
    so that any morphism $f_i$ is a pushout of the form
    $$
    \begin{tikzcd}
    T_M(S^m_k) 
    \ar[r] \ar[d, hookrightarrow]
    & A_i
    \ar[d]
    \\
    T_M(D^{m+1}_k)
    \ar[r]
    & A_{i+1}
    \end{tikzcd}
    $$
    for some $m \leq n-1$ and $k \in K$;
    \item an object $A \in \catofalg{\categ P}[\mathrm{W}^{-1}]$ belongs to $\categ{Cart}_{P, \leq n}$ if and only if it is equivalent to the image of an object in the category $\catofalg{\categ P}$ of the form $T_P (X)$
    where $X$ is the sum of a finite set of objects $S^m_k$ for some $m \leq n$ and some $k \in K$.
\end{enumerate}

\subsection{A word about formal moduli problems}

We end this section by informally describing  relations between the cellular presentation introduced above and formal moduli problems.
Let $\operad P$ be an admissible $\Sigma$-cofibrant $K$-coloured operad in chain complexes.

The operad $\operad P$ is called augmented if the unit morphism of operads $1_K \to \operad P$ has a left inverse $\epsilon$. This yields an adjunction
$$
\begin{tikzcd}
{\catofalg{\categ{P}}[\mathrm{W}^{-1}]} \ar[rr, shift left, "\epsilon_!"]
&& \categ{Ch}_R^K .
\ar[ll, shift left, "\epsilon^\ast"]
\end{tikzcd}
$$
Koszul duality techniques provide us with an operad $\categ P^!$ so that the left adjoint functor $\epsilon_!$ factorises through a functor from $\categ P$-algebras to $\categ P^!$-coalgebras called the Bar construction. Using also linear duality that sends $\categ P^!$-coalgebras to $\categ P^!$-algebras, we obtain the following sequence of functors
$$
{\catofalg{\categ{P}}[\mathrm{W}^{-1}]} 
\xrightarrow{\mathrm{Bar}}
 \catofcog{\categ{P}^!}[\mathrm{W}^{-1}]
\to
\catofalg{\categ{P}^!}[\mathrm{W}^{-1}]^\op .
$$
In some good cases described in \cite{Lurie11bis}, \cite{Pridham}, \cite{BrantnerMathew} and \cite{CalaqueCamposNuiten}, the composite functor induces a contravariant equivalence between the full subcategory $\categ{Cell}_{P,\leq O}$ of ${\catofalg{\categ{P}}[\mathrm{W}^{-1}]}$ spanned by finitely $0$-cellular objects and a full subcategory
$\categ{Art}_{P^!}$ of $\catofalg{\categ{P}^!}[\mathrm{W}^{-1}]$ spanned by objects that are called artinian and that are obtained using finitely many square zero extensions starting from the zero algebra.

If $R$ is a characteristic zero field, such results may also be proven using techniques specific to operads in chain complexes over such a field as described for instance in \cite{LodayVallette12}.

\appendix

\section{Chain complexes in non negative degrees}

In this appendix, we describe a presentation of the $\infty$-category of chain complexes in non negative degrees and of some $\infty$-categories of algebras in such chain complexes. 

\subsection{Projectively generated infinity categories}

We recall here a result of Lurie about projectively generated $\infty$-categories that is a special case of the Theorem \ref{thmmain}.

Let $\categ D$ be a small $\infty$-category with finite coproducts and let us consider an adjunction
$$
\begin{tikzcd}
\Psh_{\Sigma}(\categ D) \ar[rr, shift left, "T_M"]
&& \categ A , \ar[ll, shift left, "U^M"]
\end{tikzcd}
$$
that is strongly monadic and $\omega$-accessible. Moreover, let $\categ B$ be the full subcategory of $\categ A$ spanned by objects that are images of the functor $\categ D \hookrightarrow \Psh_{\Sigma}(\categ D) \to \categ A$. In particular $\categ B$ is stable under finite coproducts.

\begin{proposition}\cite[Corollary 4.7.3.18]{Lurie18}\label{propositionconnectivecase}
The functor $\categ A \to \Psh_{\Sigma}(\categ B)$ is an equivalence.
\end{proposition}

\begin{proof}
We are in a the context of Theorem \ref{thmmain} where $\categ C = \Psh_{\Sigma}(\categ D)$ and
$\categ E = \Psh_{\Sigma}(\categ B)$. Thus it suffices to prove that the functor $\Psh_{\Sigma}(\categ B) \to \Psh_{\Sigma}(\categ D)$
is conservative. This functor fits in the following square diagram
$$
\begin{tikzcd}
\Psh(\categ B)
\ar[r]
& \Psh(\categ D)
\\
\Psh_{\Sigma}(\categ B)
\ar[r] \ar[u]
& \Psh_{\Sigma}(\categ D) .
\ar[u]
\end{tikzcd}
$$
Since the functor $\categ D \to \categ B$ is surjective on objects, then the induced functor $\Psh(\categ B)
\to \Psh(\categ D)$ is conservative. Moreover, the two vertical functors are conservative. Thus the functor
$\Psh_{\Sigma}(\categ B) \to \Psh_{\Sigma}(\categ D)$
is also conservative.
\end{proof}

\subsection{Chain complexes}

\begin{definition}
Let $\categ{Chain}_R^{\geq 0}$ be the category of chain complexes in non negative degrees and let $\categ{Chain}_R^{\geq 0, K}$ be its product over a set $K$
$$
\categ{Chain}_R^{\geq 0, K} = \prod_{k \in K} \categ{Chain}_R^{\geq 0}.
$$
We denote $R^n_k$ for $n\geq 0$ and $k \in K$ the object of this category whose $k$-component is $R^n$ and whose $k'$-component is zero for any other $k' \neq k$.
\end{definition}

\begin{proposition}\cite{Hovey99}
The category $\categ{Chain}_R^{\geq 0, K}$ admits a model structure so that
\begin{itemize}
    \itemt weak equivalences are given by colourwise quasi-isomorphisms;
    \itemt fibrations are colourwise degreewise surjections in positive degrees;
    \itemt cofibrations are morphisms whose image in $\categ{Chain}_R^K$ is a cofibration.
\end{itemize}
\end{proposition}

\begin{definition}
Let $\categ{Ch}_R^{\geq 0, K}$ the localisation of the category $\categ{Chain}_R^{\geq 0, K}$ at weak equivalences:
$$
    \categ{Ch}_R^{\geq 0, K} = \categ{Chain}_R^{\geq 0, K}[\mathrm{W}^{-1}].
$$
\end{definition}

From Lemma \ref{lemmalocalisationproduct}, we know that $\categ{Ch}_R^{\geq 0, K}$ is canonically equivalent to the product over $K$ of the $\infty$-category $\categ{Ch}_R^{\geq 0}=\categ{Ch}_R^{\geq 0, \ast}$.

\begin{definition}
Let $\categ{Cart}^{(K)}_R$ be the full subcategory of $\categ{Chain}_R^{\geq 0, K}$ spanned by finite sums of the objects $R^n_k$ for $n>0$ and $k \in K$. In particular, it is stable through finite coproducts.
\end{definition}

\begin{proposition}
The $\infty$-category $\categ{Ch}_R^{\geq 0}$ is  equivalent to $\Psh_\Sigma(\categ{Cart}^{(K)}_R)$.
\end{proposition}

\begin{proof}
Let $\categ{Fun}_\Sigma\left((\categ{Cart}^{(K)}_R)^\op, \categ{sSet}\right)$ the full subcategory of $\Fun{(\categ{Cart}^{(K)}_R)^\op}{\categ{sSet}}$ spanned by functors that preserves finite products.
We know from \cite{Quillen67} that it admits a combinatorial model structure whose weak equivalences are objectwise Kan-Quillen weak equivalence and whose fibrations are objectwise Kan-Quillen fibrations. Moreover, we know from \cite[Corollary 5.5.9.3]{Lurie17} that the
resulting $\infty$-category is equivalent to $\Psh_\Sigma(\categ{Cart}^{(K)}_R)$.

To conclude, the Dold Kan correspondence, gives us
equivalences of model categories
$$
\categ{Fun}_\Sigma\left((\categ{Cart}^{(K)}_R)^\op, \categ{sSet}\right) \simeq \catofmod{R}^{\Delta^\op \times K} \simeq \categ{Chain}_R^{\geq 0, K} .
$$
\end{proof}

\subsection{Algebras}

Let us consider an adjunction.
$$
\begin{tikzcd}
\categ{Ch}_{R}^{\geq 0,K} \ar[rr, shift left, "T_M"]
&& \categ{A} . \ar[ll, shift left, "U^M"]
\end{tikzcd}
$$
that is $\omega$-accessible and strongly monadic,
and let $\categ{Cart}_{M}$ be the smallest full subcategory $\categ A$ that contains the objects  $T_M(R_k)$ and that is stable through finite coproducts.

\begin{corollary}
The $\infty$-category $\categ A$ is canonically equivalent to $\Psh_\Sigma(\categ{Cart}_{M})$.
\end{corollary}

\begin{proof}
This is just an application of Proposition \ref{propositionconnectivecase}.
\end{proof}

Moreover, given a $K$-coloured operad in chain complexes in non negative degrees $\categ P$, one has an adjunction between categories
$$
\begin{tikzcd}
\categ{Chain}_R^{\geq 0, K} \ar[rr, shift left,"T_P"]
&&\catofalg{\categ P} . \ar[ll, shift left,"U^P"]
\end{tikzcd}
$$
We say that $\categ P$ is connectively admissible if 
the category of $\operad P$-algebras admits a model structure transferred from that of $\categ{Chain}_R^{\geq 0, K}$, that is whose weak equivalences and fibrations are morphisms whose image through $U^P$ is respectively a weak equivalence and a fibration. In this case, the adjunction $T_P \dashv U^P$ becomes a Quillen adjunction.

\begin{proposition}
If $\operad P$ is $\Sigma$-cofibrant and connectively admissible, then the adjunction between the localised $\infty$-categories $\categ{Chain}_R^{\geq 0, K}[\mathrm W^{-1}]$ and $\catofalg{\categ P}[\mathrm W^{-1}]$ resulting from the Quillen adjunction $T_P \dashv U^P$ is $\omega$-accessible and strongly monadic.
\end{proposition}

\begin{proof}
The fact the functor induced by $U^P$ is conservative follows from the same arguments as those used in the proof of Proposition \ref{propositionconservative}. The fact that it preserves sifted colimits follows from \cite[Appendix A]{HarpazNuitenPrasma}.
\end{proof}

\bibliographystyle{amsalpha}
\bibliography{bib}

\end{document}

%% file: macros.tex
\theoremstyle{plain}
\newtheorem{theorem}{Theorem}
\newtheorem*{theorem*}{Theorem}
\newtheorem{lemma}{Lemma}
\newtheorem{proposition}{Proposition}
\newtheorem*{proposition*}{Proposition}
\newtheorem{corollary}{Corollary}

\theoremstyle{definition}
\newtheorem{definition}{Definition}
\newtheorem{hypothesis}{Hypothesis}

\theoremstyle{remark}
\newtheorem{remark}{\sc Remark}
\newtheorem{example}{\sc Example}
\newtheorem*{notation}{\sc Notation}

\newcommand{\categ}[1]{\mathsf{#1}}

\newcommand{\operad}[1]{\mathsf{#1}}

\newcommand{\catofmod}[1]{{#1}\mathrm{-}\mathsf{mod}}
\newcommand{\catofcog}[1]{\operad{#1}\mathrm{-}\mathsf{cog}}

\newcommand{\catofalg}[1]{\operad{#1}\mathrm{-}\mathsf{alg}}

\newcommand{\Fun}[2]{\mathrm{Fun}\left(#1,#2\right)}

\newcommand{\id}{\mathrm{Id}}

\newcommand{\Mapp}[3]{\mathrm{Map}_{#1}\left(#2,#3\right)}

\newcommand{\op}{\mathrm{op}}

\newcommand{\Ind}{\mathsf{Ind}}
\newcommand{\Psh}{\mathsf{Psh}}

\newcommand{\itemt}{\item[$\triangleright$]}

\newcommand{\poubelle}[1]{}